\documentclass[12pt]{article}
\usepackage{enumerate}
\usepackage{amsmath}
\usepackage{amsthm}
\usepackage{amssymb}
\usepackage[bookmarksnumbered,  plainpages]{hyperref}

\newtheorem{thm}{Theorem}[section]
\newtheorem{ex}{Example}[section]

\newtheorem{lem}{Lemma}[section]
\newtheorem{mthd}{Method}[section]

\theoremstyle{definition}
\newtheorem{defn}{Definition}[section]
\theoremstyle{remark}

\begin{document}
	
	\title{More on Projected Type  Iteration  Method and Linear Complementarity Problem}
	
	\author{Bharat Kumar$^{a,1}$, Deepmala$^{a,2}$ and A.K. Das $^{b,3}$\\
		\emph{\small $^{a}$Mathematics Discipline,}\\
		\emph{\small PDPM-Indian Institute of Information Technology, Design and Manufacturing,}\\
		\emph{\small Jabalpur - 482005 (MP), India}\\
		\emph{\small $^{b}$Indian Statistical Institute, 203 B.T. Road, }\\
		\emph{\small Kolkata - 700108, India}\\
		\emph{\small $^1$Email:bharatnishad.kanpu@gmail.com , $^2$Email: dmrai23@gmail.com}\\
		\emph{\small $^3$Email: akdas@isical.ac.in}}
	\date{}
	\maketitle
	
	\abstract{\noindent 	In this article, we establish a class of new projected type iteration methods based on matrix spitting for solving the linear complementarity problem. Also, we provide a sufficient condition for the convergence analysis when the system matrix is an $H_+$-matrix. We show the efficiency of the proposed method by using two numerical examples for different parameters.}

	
	
	
	\maketitle
	
\section{Introduction}\label{sec1}
The LCP frequently appears in an extensive range of applications that include scientific computing and engineering, such as the free boundary problem and the Nash equilibrium point of the bimatrix game; the American option pricing problem; mathematical economics; operations research; control theory; optimization theory; stochastic optimal control; economics; and elasticity theory.  For details see  \cite{Cottle1992}, \cite{Neogy2008}, \cite{Murty1988} \cite{jana2021more}, \cite{dutta2022column}, \cite{neogy2011singular}, \cite{das2018some}, \cite{jana2018processability}, \cite{das2023more} and \cite{kumar2022error}.\\
Assuming $\mathcal{A}\in \mathcal{R}^{n\times n}$ and  a vector  $\,\sigma\,\in\,\mathcal{R}^{n}.$ The  linear complementarity problem denoted as LCP$(\sigma,\mathcal{A})$ is to find the solution  $\lambda \in \mathcal{R}^{n}$ to the following system 
\begin{eqnarray}\label{eq1}
	\lambda\geq 0, ~~~~    \mathcal{A}\lambda +\sigma \geq 0,~~~~ \lambda^T(\mathcal{A}\lambda +\sigma)=0	
\end{eqnarray}
The methods for solving linear complementarity problems are divided into two categories: the pivoting method \cite{Das2014} \cite{Das2016}, \cite{Jana2019} and iterative method \cite{Najafi2013i}, \cite{Hadjidimos2018}, \cite{Kumar2022}, \cite{Kumar2023} and \cite{Kumar2023i}. Lemke and Howson \cite{lmke} introduced the complementary pivot method, but some matrices are not processable by this method as well as by Lemke’s Method. 
\noindent The linear complementarity problem can be solved in a number of ways by an iterative process; namely, the projected type methods \cite{Bai1997}, \cite{Hadjidimos2018},  \cite{Najafi2013i}, the modulus method  \cite{Bai2010}, \cite{Dong2009} and the modulus based matrix splitting iterative methods  \cite{Liu2016} and \cite{Zheng2017}.\\
Fang proposed a general fixed point method (GFP) \cite{Xi2021} assuming the case where $\Omega = \omega A_{\mathcal{D}}^{-1}$ with $\omega \textgreater 0$ and $A_{\mathcal{D}}$ is the diagonal matrix of $\mathcal{A}$. The GFP approach takes less iterations than the modulus-based successive over-relaxation (MSOR) \cite {Bai2010} iteration method. However, the GFP approach calculates the numerical solution component by component of vectors, which takes a long time.\\ In this article, we present a class of new projected type iteration methods by using the ideas of Xi \cite{Xi2021} and Ali \cite{rashid2022}. Also, we show that the fixed point equation and the linear complementarity problem are equivalent, discuss convergence conditions and provide a convergence domain for our proposed method.\\
The article is organized as follows: some required definitions, notations and well-known lemmas are given in Section 2, which will be used for the discussions in the remaining sections of this work. New projected type iteration methods are constructed in Section 3 with the help of the new equivalent fixed point form of the LCP$(\sigma, \mathcal{A})$. In Section 4, we establish the convergence domain of our proposed method. A numerical comparison between the proposed methods and modulus-based matrix splitting iteration methods, introduced by Bai \cite{Bai2010}, is illustrated in Section 5. Section 6 contains the conclusion of the article.
\section{Preliminaries}\label{Preli}
In this section, we provide an overview of various essential notations, definitions, and foundational results.\\
Suppose  $ \mathcal{A}=( {a}_{ij}) \in \mathcal{R}^{n\times n}$  and $ \mathcal{B}=( {b}_{ij}) \in \mathcal{R}^{n\times n}$ are real square matrices.  For $\mathcal{A}=({a}_{ij}) \in \mathcal{R}^{n\times n}$ and $\mathcal{B}=( {b}_{ij}) \in \mathcal{R}^{n\times n}$, $\mathcal{A} \geq $ $(\textgreater)$ $ \mathcal{B}$ means ${a}_{ij}\geq (\textgreater)$ $ {b}_{ij}$ for all $ i,j$.

\begin{defn}\cite{Xi2021}	Let  $ \mathcal{A}=({a}_{ij})\in \mathcal{R}^{n\times n}$. Then $| \mathcal{A}|=({c}_{ij})$  is defined by $ {c}_{ij}  = | {a}_{ij}|$ $\forall ~i,j$ and $| \mathcal{A}| $ represent that $ {a}_{ij} \geq 0$ $\forall ~i,j $.
\end{defn} 

\begin{defn}\label{defn0}\cite{Xi2021}
	Let $\mathcal{A}, \mathcal{B} \in \mathcal{R}^{n \times n}$. Then $| \mathcal{A}+\mathcal{B}| \leq | \mathcal{A}| +| \mathcal{B}| $  and $| \mathcal{A}\mathcal{B}| \leq | \mathcal{A}| | \mathcal{B}| $. Moreover   ${x}, {y} \in \mathcal{R}^{n}$ then $| {x}+{y}| \leq | {x}| +| {y}| $ and  $| | {x}| -| {y}| | \leq | {x}-{y}| $.
\end{defn}

\begin{defn}\cite{Das2016}
	Let  $\mathcal{A}\in \mathcal{R}^{n\times n}$.  $\mathcal{A}$ is said to be a $P$-matrix if all its principle minors are positive i.e.  det$({\mathcal{A}}_{\gamma \gamma}) ~\textgreater~ 0$ for all $\gamma \subseteq \{1,2,\ldots, n\}$.
\end{defn}
\begin{defn}\cite{Xi2021}
	Suppose  $ \mathcal{A}\in  \mathcal{R}^{n\times n}$. Then its  comparison matrix  is defined as $\langle {a}_{ij}\rangle=|{a}_{ij}| $ if $i=j$ and $\langle {a}_{ij}\rangle=-| {a}_{ij}|  $ if $i \neq j$.
\end{defn}

\begin{defn}\label{def1}\cite{Frommer1992}
	Suppose  $ \mathcal{A} \in  \mathcal{R}^{n\times n}$. $\mathcal{A}$ is said to be a $Z$-matrix if all of its non-diagonal elements are less  than or equal to zero;  $\mathcal{A}$ is  said to be an  $M$-matrix if $\mathcal{A}^{-1}\geq 0$ as well as   $Z$-matrix;
	$ \mathcal{A}$  is said to be  an $H$-matrix if   $\langle \mathcal{A} \rangle$ is an $M$-matrix; $\mathcal{A}$ is an  $H_+$-matrix if it is an $H$-matrix with  $ {a}_{ii} ~\textgreater~ 0 ~\forall ~i \in \{1,2,\ldots,n\}$.	
\end{defn} 
\begin{defn}\label{def2}$\cite{Frommer1992}$
	Suppose  $\mathcal{A} \in \mathcal{R}^{n\times n}$.	The splitting $\mathcal{A} = \mathcal{M}-\mathcal{N} $ is called an $M$-splitting if $\mathcal{M}$ is a nonsingular $M$-matrix and $\mathcal{N} \geq 0$; an  $H$-splitting if $\langle \mathcal{M} \rangle -| \mathcal{N}| $ is an $M$-matrix; an $H$-compatible splitting if  $\langle \mathcal{A} \rangle =  \langle \mathcal{M} \rangle - | \mathcal{N}| $.
\end{defn} 
\begin{lem}\cite{rashid2022}\label{lem0}
	Let ${x}, y \in \mathcal{R}^{n} $.  ${x}\geq 0$, ${y}\geq 0$, $x^{T}{y}=0$ if and only if  ${x}+{y}=| {x}-{y}| $.
\end{lem}
\begin{lem}\label{lem1}\cite{Frommer1992}
	Suppose  $\mathcal{A}, B_{1} \in \mathcal{R}^{n\times n}$. If $\mathcal{A}$ and $\mathcal{B}$ are $M$ and $Z$-matrices respectively with $\mathcal{A} \leq \mathcal{B}$ then $\mathcal{B}$ is an $M$-matrix. 
	If $\mathcal{A}$ is an $H$-matrix then $| \mathcal{A}^{-1}| \leq \langle \mathcal{A}\rangle^{-1}$.
	If $\mathcal{A} \leq \mathcal{B}$, then $\rho(\mathcal{A}) \leq \rho(\mathcal{B})$.
\end{lem}
\begin{lem}\label{lem2}\cite{Xi2021}
	Let $\mathcal{A}\in \mathcal{R}^{n\times n}$ be  an $M$-matrix  and $\mathcal{A}=\mathcal{M}-\mathcal{N}$ be an $M$-splitting. Let $\rho$ be the spectral radius, then ~$\rho(\mathcal{M}^{-1}\mathcal{N})$ $\textless $ $1$.
\end{lem}	
\begin{lem}\label{lem3} \cite{Berman1994} 	Suppose  $\mathcal{A} \in \mathcal{R}^{n\times n}$ with splitting $\mathcal{A} = \mathcal{M}-\mathcal{N} $. Let splitting be an $H$-compatible of an $H$-matrix, then it is an $H$-splitting but converse is not true.
\end{lem}
\begin{lem}\label{lem4} \cite{Frommer1992}
	Suppose  $\mathcal{A} \geq 0 \in \mathcal{R}^{n \times n}$, if there exists $v ~\textgreater~ 0 \in \mathcal{R}^{n}$ and a scalar $\alpha_{1} ~\textgreater ~ 0$  such that $\mathcal{A}v \leq  \alpha_{1} v$ then $\rho(\mathcal{A}) \leq \alpha_{1} $. Moreover, if $ \mathcal{A}v~ \textless ~v$ then $\rho(\mathcal{A})\textless  1$.
\end{lem} 
\section{Main results}			
For a given vector $\zeta \in \mathcal{R}^n $, we indicate the vector $\zeta _{+}=$max$\{0,\zeta\}$ and matrix $\mathcal{A}=(\mathcal{M}+I+D_{\mathcal{A}})-(\mathcal{N}+I+D_{\mathcal{A}})$, where $D_{\mathcal{A}}$ is  diagonal matrix of $\mathcal{A}$. In the following result, we convert the LCP $(\sigma, \mathcal{A})$ into a fixed point formulation.
\begin{thm}\label{thm1} Let $\mathcal{A}\in \mathcal{R}^{n\times n} $ with the splitting  $\mathcal{A}= (\mathcal{M}+I+D_{\mathcal{A}})- (\mathcal{N}+I+D_{\mathcal{A}})$. Let $\lambda=\zeta_{+}$, then equivalent formulation of  the LCP$(\sigma,\mathcal{A})$ in form of fixed point equation is \begin{eqnarray}\label{eq2}
		\zeta_{+}=(\mathcal{M}+2I+D_{\mathcal{A}})^{-1}[(\mathcal{N}+I+D_{\mathcal{A}})\zeta_{+}+ |(\mathcal{A} - I)\zeta_{+}+ \sigma|- \sigma]  
	\end{eqnarray}
\end{thm}
\begin{proof}
	We have  $\lambda=\zeta_{+} \geq 0 $ and $\mathcal{A}\lambda+\sigma \geq 0$, from Lemma $\ref{lem0}$,
	\begin{equation*}
		\begin{split}
			( \mathcal{A}\zeta_{+}+\sigma +\zeta_+)&=| \mathcal{A}\zeta_{+}+\sigma-\zeta_+|\\
			(I+\mathcal{A})\zeta_+&=| (\mathcal{A}-I)\zeta_{+}+\sigma| -\sigma\\
			(\mathcal{M}+2I+D_{\mathcal{A}})\zeta_+&=(\mathcal{N}+I+D_{\mathcal{A}})\zeta_{+}+| (\mathcal{A}-I)\zeta_{+}+\sigma| -\sigma,\\
		\end{split}
	\end{equation*}
	the above equation can be rewritten as,
	\begin{equation}\label{eq2'}
		\zeta_+=(\mathcal{M}+2I+D_{\mathcal{A}})^{-1}[(\mathcal{N}+I+D_{\mathcal{A}})\zeta_{+}+| (\mathcal{A}-I)\zeta_{+}+\sigma| -\sigma]
	\end{equation}
\end{proof}

	\noindent In the following, Based on Equation ($\ref{eq2}$), we propose an iteration method which is known as Method 3.1 to solve the LCP$(\sigma, \mathcal{A})$.
	\begin{mthd}\label{mthd1}
			Let $\mathcal{A}=(\mathcal{M} +I+D_{\mathcal{A}})-(\mathcal{N}+I+D_{\mathcal{A}})$ be a splitting of the matrix $\mathcal{A}\in \mathcal{R}^{n\times n}$ and the matrix $(\mathcal{M}+2I+D_{\mathcal{A}})$ be the nonsingular. Then we use
		the following equation for Method 3.1 is
		\begin{equation}\label{eq003}
				\zeta^{(\eta+1)}_{+}=(\mathcal{M}+2I+D_{\mathcal{A}})^{-1}[(\mathcal{N}+I+D_{\mathcal{A}})\zeta^{(\eta)}_{+}+ |(\mathcal{A} - I)\zeta^{(\eta)}_{+}+ \sigma| - \sigma]
		\end{equation}
		Let Residual be the Euclidean norm of the error vector, which is defined as follows:  $$ Res(\lambda^{(\eta)})=|min(\lambda^{(\eta)}, \mathcal{A}\lambda^{(\eta)}+\sigma) |_{2}.$$ Consider a nonnegative initial vector $\lambda^{(0)}\in \mathcal{R}^n$.   The iteration process continues until the iteration sequence $\{\lambda^{(\eta)}\}_{\eta=0}^{+\infty} \subset \mathcal{R}^n$ converges.   For  $\eta=0,1,2,\ldots$, the iterative
		process continues until the iterative sequence  $\lambda^{(\eta+1)}\in \mathcal{R}^{n}$ converges. The iteration process stops if $Res(\lambda^{(\eta)})$ $\textless $ $ \epsilon $. For computing $\lambda^{(\eta+1)}$ we use the following  steps.\\
		\textbf{Step 1}: Given an initial vector $\zeta^{(0)} \in \mathcal{R}^{n}$,  $\epsilon ~\textgreater ~ 0 $  and set $ \eta=0 $.\\
		\textbf{Step 2}: Using the following scheme, generate the sequence $\lambda^{(\eta)}$:
		\begin{equation*}
			\zeta^{(\eta+1)}_{+}=(\mathcal{M}+2I+D_{\mathcal{A}})^{-1}[(\mathcal{N}+I+D_{\mathcal{A}})\zeta^{(\eta)}_{+}+ |(\mathcal{A} - I)\zeta^{(\eta)}_{+}+ \sigma| - \sigma],
		\end{equation*}
		and set $\lambda^{(\eta+1)}=\zeta_{+}^{(\eta+1)}$, where $\zeta_{+}^{(\eta+1)}$ is the $(\eta+1)^{th}$ approximate solution of Equation ($\ref{eq2'}$) .\\
		\textbf{Step 3}: If $ Res(\lambda^{(\eta)})$ $ \textless $ $\epsilon$  then stop; otherwise, set $\eta=\eta+1$ and return to step 2. 
	\end{mthd}
 \noindent Moreover, Method \ref{mthd1} provides a general structure for solving LCP$(\sigma, \mathcal{A})$. We obtain a class of new projected type iteration relaxation methods using  matrix splitting. We express the system matrix $\mathcal{A}=(\mathcal{M} +I+D_{\mathcal{A}})-(\mathcal{N}+I+D_{\mathcal{A}})$.  Then
\begin{enumerate}
	\item when $\mathcal{M}=D_{\mathcal{A}}-L_{\mathcal{A}}$ and $\mathcal{N}=U_{\mathcal{A}}$, Equation (\ref{eq003}) gives the new projected type Gauss Seidel iteration (NPGS)   method
	\begin{eqnarray*}
		\begin{split}
			\zeta^{(\eta+1)}_{+}&=(D_{\mathcal{A}}-L_{\mathcal{A}}+2I+D_{\mathcal{A}})^{-1}[(U_{\mathcal{A}}+I+D_{\mathcal{A}})\zeta^{(\eta)}_{+}\\&+ | (\mathcal{A} - I)\zeta^{(\eta)}_{+}+ \sigma |   - \sigma]. 
		\end{split}
	\end{eqnarray*}
	\item when $\mathcal{M}=(\frac{1}{\alpha_1}D_{\mathcal{A}}-L_{\mathcal{A}})$ and $\mathcal{N}=(\frac{1}{\alpha_1}-1)D_{\mathcal{A}}+U_{\mathcal{A}}$, Equation (\ref{eq003}) gives the new projected type successive   overrelaxation   iteration (NPSOR)  method
	\begin{eqnarray*}
		\begin{split}
			\zeta^{(\eta+1)}_{+}&=(D_{\mathcal{A}}-\alpha_1 L_{\mathcal{A}}+\alpha_1 (2I+D_{\mathcal{A}}))^{-1}[((1-\alpha_1)D_{\mathcal{A}}+U_{\mathcal{A}}\\&+\alpha_1I+D_{\mathcal{A}})\zeta^{(\eta)}_{+}+ \alpha_1|(\mathcal{A} - I)\zeta^{(\eta)}_{+}+ \sigma| - \alpha_1 \sigma].  
		\end{split}
	\end{eqnarray*}
	\item when  $\mathcal{M}=(\frac{1}{\alpha_1})(D_{\mathcal{A}}-\beta_1 L_{\mathcal{A}})$ and $\mathcal{N}=(\frac{1}{\alpha_1})[(1-\alpha_1)D_{\mathcal{A}}+(\alpha_1 - \beta_1 )L_{\mathcal{A}}+\alpha_1 U_{\mathcal{A}}]$, Equation (\ref{eq003}) gives the new projected type accelerated  overrelaxation   iteration (NPAOR)   method
	\begin{eqnarray*}
		\begin{split}
			\zeta^{(\eta+1)}_{+}&=(D_{\mathcal{A}}-\beta_1 L_{\mathcal{A}}+\alpha_1 (2I+D_{\mathcal{A}}))^{-1}[((1-\alpha_1)D_{\mathcal{A}}+(\alpha_1-\beta_1)U_{\mathcal{A}}\\&+\alpha_1I+D_{\mathcal{A}})\zeta^{(\eta)}_{+}+ \alpha_1| (\mathcal{A} - I)\zeta^{(\eta)}_{+}+ \sigma| - \alpha_1 \sigma]. 
		\end{split} 
	\end{eqnarray*}
	
\end{enumerate}
When $(\alpha_1, \beta_1)$ takes the values $(\alpha_1, \alpha_1)$, $(1, 1)$ and $(1, 0)$, the NPAOR method transforms into the new projected type successive overrelaxation (NPSOR), new projected type Gauss-Seidel (NPGS) and new projected type Jacobi (NPJ) methods respectively.

\section {Convergence analysis}
\noindent	In the following, we present  the convergence condition when the system matrix  $\mathcal{A}$ of LCP$(\sigma, \mathcal{A})$ is a $P$-matrix.
\begin{thm}\label{thm2}
	Let $\mathcal{A} \in \mathcal{R}^{n\times n}$ be a $P$-matrix and  $\zeta_{+}^{*}$ be the solution  of  Equation $(\ref{eq2})$. Let $	\rho(|(\mathcal{M}+2I+D_{\mathcal{A}})^{-1}|(| \mathcal{N}+I+D_{\mathcal{A}}| +| \mathcal{A}-I| ))~\textless ~1 $. Then the sequence $\{\zeta_+^{(\eta)}\}^{+\infty}_{\eta=1}$  generated by Method  $\ref{mthd1}$ converges to the solution $\zeta_+^{*}$ for any initial vector $\zeta^{(0)}\in \mathcal{R}^{n}$.	
\end{thm}
\begin{proof}
	Let $\zeta^*_{+}$ be the solution of Equation $(\ref{eq2})$, then error is 
	\begin{equation*}
		\begin{split}
			(\mathcal{M}+2I+D_{\mathcal{A}})(\zeta^{(\eta+1)}_{+}-\zeta^*_{+})&=(\mathcal{N}+I+D_{\mathcal{A}})(\zeta^{(\eta)}_{+}-\zeta^*_{+})+ |(\mathcal{A}-I)\zeta^{(\eta)}_{+}+\sigma|\\&-|(\mathcal{A}-I)\zeta^{*}_{+}+\sigma|\\
			|(\mathcal{M}+2I+D_{\mathcal{A}})(\zeta^{(\eta+1)}_{+}-\zeta^*_{+})|&=|(\mathcal{N}+I+D_{\mathcal{A}})(\zeta^{(\eta)}_{+}-\zeta^*_{+})+ |(\mathcal{A}-I)\zeta^{(\eta)}_{+}+\sigma|\\&-|(\mathcal{A}-I)\zeta^{*}_{+}+\sigma||\\
		\end{split}	
	\end{equation*} 
	\begin{equation*}
		\begin{split}
			&\leq |(\mathcal{N}+I+D_{\mathcal{A}})(\zeta^{(\eta)}_{+}-\zeta^*_{+})+ |(\mathcal{A}-I)(\zeta^{(\eta)}_{+}-\zeta^*_{+})||\\
			&\leq |(\mathcal{N}+I+D_{\mathcal{A}})||(\zeta^{(\eta)}_{+}-\zeta^*_{+})|+ |(\mathcal{A}-I)||(\zeta^{(\eta)}_{+}-\zeta^*_{+})|\\
			|(\zeta^{(\eta+1)}_{+}-\zeta^*_{+})|	&\leq |(\mathcal{M}+2I+D_{\mathcal{A}})^{-1}|[|(\mathcal{N}+I+D_{\mathcal{A}})|+ |(\mathcal{A}-I)|]|(\zeta^{(\eta)}_{+}-\zeta^*_{+})|\\ 
			|(\zeta^{(\eta+1)}_{+}-\zeta^*_{+})|	&\textless |(\zeta^{(\eta)}_{+}-\zeta^*_{+})|.
		\end{split}	
	\end{equation*}
	Hence $\zeta^{(\eta)}_{+}$ converges to the solution $\zeta^*_{+}$.
\end{proof}
\noindent  Now we discuss the convergence conditions  for  Method  \ref{mthd1}  when the system matrix $\mathcal{A}$ of LCP$(\sigma,\mathcal{A})$ is an $H _+$-matrix.
\begin{thm}{\label{thm6}}
	Let $\mathcal{A}\in \mathcal{R}^{n\times n}$ be an $H_{+ }$-matrix and   $\mathcal{A}=\mathcal{M}-\mathcal{N}=(\mathcal{M}+I+D_{\mathcal{A}})-(\mathcal{N}+I+D_{\mathcal{A}})$ be an $H$-compatible splitting of the matrix $\mathcal{A}$, such that $\langle \mathcal{A} \rangle=\langle \mathcal{M}\rangle -| \mathcal{N}| $= $\langle \mathcal{M} + I+D_{\mathcal{A}} \rangle -| \mathcal{N}+I+D_{\mathcal{A}}| $ and  either one of the following conditions hold:\\
	(1) $D_{\mathcal{A}} ~\geq ~I$ and $\langle \mathcal{A} \rangle+2I-D_{\mathcal{A}}-| B|$ is an $M$- matrix, where $B= L_{\mathcal{A}}+U_{\mathcal{A}}$;\\
	(2) $D_{\mathcal{A}}~ \textless ~I$.\\
	Then the sequence $\{\zeta_+^{(\eta)}\}^{+\infty}_{\eta=1}$  generated by Method  $\ref{mthd1}$ converges to the solution $\zeta_+^{*}$ for any initial vector $\zeta^{(0)}\in \mathcal{R}^{n}$.
\end{thm}
\begin{proof} Let $\mathcal{A}=\mathcal{M}-\mathcal{N}=(\mathcal{M}+I+D_{\mathcal{A}})-(\mathcal{N}+I+D_{\mathcal{A}})$ and it holds that 
	$\langle \mathcal{A} \rangle\leq  \langle \mathcal{M} + I+D_{\mathcal{A}} \rangle \leq diag(\mathcal{M}+I+D_{\mathcal{A}})$, 
	hence $(\mathcal{M}+I+D_{\mathcal{A}})$ is an $H_{+}$-matrix and it holds that 
	$$ | (\mathcal{M}+2I+D_{\mathcal{A}})^{-1}|  \leq (\langle \mathcal{M} \rangle+2I+D_{\mathcal{A}})^{-1}.$$
	Let $T=| (\mathcal{M}+2I+D_{\mathcal{A}})^{-1}| (| (\mathcal{N}+I+D_{\mathcal{A}})| + | (\mathcal{A}-I)|)$.\\
	\noindent Then
	\begin{equation*}
		\begin{split}
			T & =| (2I+\mathcal{M}+D_{\mathcal{A}})^{-1}| [| \mathcal{N}+I+D_{\mathcal{A}}| +| (\mathcal{A}-I)| ]\\
			&\leq (2I+\langle \mathcal{M}\rangle+D_{\mathcal{A}})^{-1}[| \mathcal{N}+I+D_{\mathcal{A}}| +|(\mathcal{A}-I)|]\\
			&\leq (2I+\langle \mathcal{M}\rangle+D_{\mathcal{A}})^{-1}[| \mathcal{N}+I+D_{\mathcal{A}}|+|(D_{\mathcal{A}}-I)-(L_{\mathcal{A}}+U_{\mathcal{A}})|]\\
			&\leq (\langle \mathcal{M}\rangle+2I+D_{\mathcal{A}})^{-1}[(\langle \mathcal{M}\rangle+2I+D_{\mathcal{A}})-(\langle \mathcal{M}\rangle+2I+D_{\mathcal{A}})\\&+| \mathcal{N}+I+D_{\mathcal{A}}| +| D_{\mathcal{A}}-I| +| L_{\mathcal{A}}+U_{\mathcal{A}}| ].\\
		\end{split}
	\end{equation*}
	Case 1. Suppose  $D_{\mathcal{A}} ~\geq ~I$ and $\langle \mathcal{A} \rangle+2I-D_{\mathcal{A}}-| B|$ is an $M$-matrix then
	\begin{equation*}
		\begin{split}
			T&\leq I-(\langle \mathcal{M}\rangle+2I+D_{\mathcal{A}})^{-1}[(\langle \mathcal{M}\rangle+2I+D_{\mathcal{A}})-| \mathcal{N}+I+D_{\mathcal{A}}| \\&-D_{\mathcal{A}}+I-| L_{\mathcal{A}}+U_{\mathcal{A}}| ]\\	
			&\leq I-(\langle \mathcal{M}\rangle+2I+D_{\mathcal{A}})^{-1}[I+(\langle \mathcal{M}\rangle+I+D_{\mathcal{A}})-| \mathcal{N}+I+D_{\mathcal{A}}| \\&-D_{\mathcal{A}}+I-| L_{\mathcal{A}}+U_{\mathcal{A}}| ]\\
			&\leq I-(\langle \mathcal{M}\rangle+2I+D_{\mathcal{A}})^{-1}[(\langle \mathcal{M}\rangle+I+D_{\mathcal{A}})-| \mathcal{N}+I+D_{\mathcal{A}}| +2I -D_{\mathcal{A}}\\&-| L_{\mathcal{A}}+U_{\mathcal{A}}| ]\\
			&\leq I-(\langle \mathcal{M}\rangle+2I+D_{\mathcal{A}})^{-1}(\langle \mathcal{A}\rangle+2I -D_{\mathcal{A}}-| B|).\\		
		\end{split}
	\end{equation*}
	Since $\langle \mathcal{A}\rangle+2I -D_{\mathcal{A}}-| B|$ is an $M$-matrix,  there exists a positive vector $v~ \textgreater ~0$ such that $$(\langle \mathcal{A}\rangle+2I -D_{\mathcal{A}}-| B|) v ~\textgreater ~0.$$
	Therefore, $$Tv\leq (I-2(\langle \mathcal{M}\rangle+2I+D_{\mathcal{A}})^{-1}(\langle \mathcal{A}\rangle+2I -D_{\mathcal{A}}-| B|)v~ \textless~  v$$
	$$Tv~\textless~ v.$$
	Case 2. Suppose   $D_{\mathcal{A}} \textless I$ 	then
	\begin{equation*}
		\begin{split}
			T &\leq I-(\langle \mathcal{M}\rangle+2I+D_{\mathcal{A}})^{-1}[(\langle \mathcal{M}\rangle+2I+D_{\mathcal{A}})-| \mathcal{N}+I+D_{\mathcal{A}}| \\&+D_{\mathcal{A}}-I-| L_{\mathcal{A}}+U_{\mathcal{A}}|]\\	
			&\leq I-(\langle \mathcal{M}\rangle+2I+D_{\mathcal{A}})^{-1}[(\langle \mathcal{M}\rangle+I+D_{\mathcal{A}})-| \mathcal{N}+I+D_{\mathcal{A}}| \\&+D_{\mathcal{A}}-| L_{\mathcal{A}}+U_{\mathcal{A}}|]\\
			&\leq I-(\langle \mathcal{M}\rangle+2I+D_{\mathcal{A}})^{-1}[(\langle \mathcal{M}\rangle+I+D_{\mathcal{A}})-| \mathcal{N}+I+D_{\mathcal{A}}| +D_{\mathcal{A}}\\&- | L_{\mathcal{A}}+U_{\mathcal{A}}| ]\\
			&\leq I-2(\langle \mathcal{M}\rangle+2I+D_{\mathcal{A}})^{-1}\langle \mathcal{A}\rangle.\\		
		\end{split}
	\end{equation*}
	Since $\langle \mathcal{A}\rangle$ is an $M$-matrix, then there exists a positive vector $v~ \textgreater ~0$ such that $$(\langle \mathcal{A}\rangle) v ~\textgreater ~0.$$
	Therefore, $$Tv\leq (I-2(\langle \mathcal{M}\rangle+2I+D_{\mathcal{A}})^{-1}(\langle \mathcal{A}\rangle)v~ \textless~  v$$
	This implies that 
	$$Tv~\textless~ v.$$
	From Cases 1 and  2 and based on Lemma $\ref{lem4}$, we obtain that $\rho(T)\textless  1$. Therefore based on Theorem  $\ref{thm2}$, the iteration sequence $\{\zeta_+^{(\eta)}\}^{+\infty}_{\eta=1}$ generated by Method  $\ref{mthd1}$ converges to $\zeta^{*}_+$ for  any initial vector $\zeta^{(0)}.$
\end{proof}
\section{Numerical examples}
In this section, two numerical examples are given in this part to demonstrate the effectiveness of our proposed method and  use  some notation as  number of iteration steps (denoted by IT), CPU time in seconds (denoted by CPU). Let   $\zeta^{(0)}=(1,0,\ldots 1,0,\ldots)^T\in \mathcal{R}^{n}$ be an initial vector and set $\epsilon=10^{-5}$. We consider the  LCP$(\sigma,\mathcal{A})$ which has always unique solution and   define $\sigma=-\mathcal{A} {\lambda}^{*}$, where 
${\lambda}^{*}=(1,2,1,\cdots,1,2)^{T} \in \mathcal{R}^n.$
\noindent The proposed new projected  Gauss Seidel iteration  method (NPGS) and the new projected  successive over relaxation iteration method (NPSOR) are compared with the modulus based Gauss Seidel (MGS) Method    and the modulus based successive over relaxation (MSOR) method   \cite{Bai2010} respectively, which are effective in solving LCP$(\sigma,\mathcal{A})$ and  set $ \Omega = \frac{1}{2\alpha}D_{\mathcal{A}}$ for MGS and MSOR methods.   
\noindent Matlab version 2021a on an Acer Desktop (Intel(R) Core(TM) i7-8700 CPU @ 3.2 GHz, 3.19 GHz, 16.00GB RAM) is used for all calculations.
Table \ref{t1} and Table \ref{t2} list the numerical results for new projected type iteration matrix splitting Method $\ref{mthd1}$ (NPGS, NPSOR) and modulus based matrix splitting Methods (MGS, MSOR).
\begin{ex}\label{ex1}
	The system matrix $\mathcal{A}$  are generated by $\mathcal{A}= P_{1}+\delta_{1}I$, 
	where $\delta_{1}$ are nonnegative real parameter and \\ \ \\
	$P_{1}=
	\begin{bmatrix}
		L_{1} & -I_{1} &0 & \ldots &0  \\
		-I_{1} & L_{1} & -I_{1} &\ldots &0  \\  
		0& -I_{1} & L_{1} &-I_{1}   &0\\ 
		0& \ldots & I_{1} & \ddots & -I_{1}\\ 
		0& \ldots & 0 &-I_{1} &L_{1}\\ 
	\end{bmatrix} $$\in \mathcal{R}^{n\times n}$,
	$L_{1}=
	\begin{bmatrix}
		4 &-1  & \ldots&\ldots &0   \\
		-1&  4 & -1 & \ldots& 0  \\  
		0&  -1& 4 &-1 &  0\\  
		0& \ldots &-1  & \ddots & -1\\ 
		0& \ldots & \ldots &-1 &4\\ 
	\end{bmatrix} $\\ \ \\ $\in \mathcal{R}^{m\times m}$, 
	 where 
	$I_{1}$ is the identity matrix of order $m$.
	\begin{table}[h!]
		\caption{Results for  MGS and MSOR methods    and   NPGS and NPSOR methods,  when  $\delta_{1}=4$.}
		\label{t1}
		\centering
		\begin{tabular}{@{}|l|l| l l l l l l|@{}}
			\hline
			&\textbf{n} &${100}$   & ${900}$   & ${2500}$ & ${3600}$& ${6400}$ & ${10000}$ \\ [.9ex] 
			\hline\hline
			$\textbf{MGS}$&$\textbf{IT} $&36 & 40 & 41 & 41& 42 & 42\\
			$\alpha=1$	& $\textbf{CPU}$ & 0.0030 & 0.0254  & 0.2550  &0.6083&1.8468&2.7943 \\
			&$\textbf{Res} $&9.7e-06 &8.0e-06 &7.9e-06 &8.9e-06&7.4e-06&8.4e-06\\
			$\textbf{NPGS}$&$\textbf{IT}$&21  &  23&254  &24&25&25\\
			$\alpha_1=1$	&$\textbf{CPU} $ & 0.0021 & 0.0035 &0.0175 &0.0636&0.1725&0.3785\\
			&$\textbf{Res} $&5.2e-06 &7.1e-06 &6.5e-06&8.0e-06& 5.5e-06&7.0e-06\\
			\hline 
			$\textbf{MSOR}$&$\textbf{IT} $&15 & 17 & 18 & 18& 18 &19\\
			$\alpha=0.85$	& $\textbf{CPU}$ & 0.0024 & 0.0044  & 0.0134 &0.0462&0.1118&0.2246 \\
			&$\textbf{Res} $&9.5e-06 &7.6e-06 &5.2e-06 &6.5e-06&8.9e-06&4.3e-06\\
			$\textbf{NPSOR}$&$\textbf{IT}$&15  &  16&17  & 17 & 17 & 17\\
			$\alpha_1=1.7$	&$\textbf{CPU} $ & 0.0019 & 0.0031 &0.0138 &0.0449&0.1108&0.2217\\
			&$\textbf{Res} $&5.9e-06 &6.8e-06 &4.2e-06&4.9e-06& 6.3e-06&7.7e-06\\
			\hline 
		\end{tabular}
	\end{table}
\end{ex} 
\begin{ex} The system matrix $\mathcal{A}\in \mathcal{R}^{n\times n}$  is generated by $\mathcal{A}= P_{1}+\delta_{1}I$, 
	where $\delta_{1}$  are nonnegative real parameter and \\ \ \\
	$P_{1}=
	\begin{bmatrix}
		L_{1} & -0.5I_{1} &0 & \ldots &0  \\
		-1.5I_{1} & L_{1} & -0.5I_{1} &\ldots &0  \\  
		0& -1.5I_{1} & L_{1} &-0.5I_{1}   &0\\ 
		0& \ldots & -1.5I_{1} & \ddots & -0.5I_{1}\\ 
		0& \ldots & 0 &-1.5I_{1} &L_{1} 
	\end{bmatrix} $$ $,
	$L_{1}=
	\begin{bmatrix}
		4 &-1  & \ldots&\ldots &0   \\
		-1&  4 & -1 & \ldots& 0  \\  
		0&  -1& 4 &-1 &  0\\  
		0& \ldots &-1  & \ddots & -1\\ 
		0& \ldots & \ldots &-1 &4\\ 
	\end{bmatrix} $ \\ \ \\ $\in \mathcal{R}^{m\times m}$,
	\noindent where $I_{1}$ is the identity matrix of order $m$. 
	\begin{table}[h!]
		\caption{Results for  MGS and MSOR methods    and   NPGS and NPSOR methods,  when  $\delta_{1}=4$.}
		\label{t2}
		\centering
		\begin{tabular}{@{}|l|l| l l l l l l|@{}}
			\hline
			&\textbf{n} &${100}$   & ${400}$   & ${900}$ & ${1600}$& ${2500}$ & ${3600}$ \\ [.9ex] 
			\hline\hline
			$\textbf{MGS}$&$\textbf{IT} $&24 & 26 & 26& 26& 27 &27\\
			$\alpha=1$	& $\textbf{CPU}$ & 0.0038 & 0.0043  & 0.0203 &0.0691&0.1978& 0.3915\\
			&$\textbf{Res} $&9.0e-06 &6.5e-06 &8.7e-06 &9.6e-06&6.6e-06&7.4e-06\\
			$\textbf{NPGS}$&$\textbf{IT}$ &18  &  21  &  22  & 22  &  22  &  23\\
			$\alpha_1=1$	&$\textbf{CPU} $ & 0.0019 & 0.0033 &0.0138 &0.0469&0.1260&0.2637\\
			&$\textbf{Res} $&9.8e-06 &5.5e-06 &4.8e-06&5.9e-06& 8.0e-06&4.9e-06\\
			\hline 
			$\textbf{MSOR}$&$\textbf{IT} $&14 & 14 & 15 & 15& 15 &15\\
			$\alpha=0.88$	& $\textbf{CPU}$ & 0.0025 & 0.0037  & 0.0121 &0.0385&0.0932&0.1823 \\
			&$\textbf{Res} $&3.8e-06 &8.2e-06 &4.0e-06 &4.6e-06&5.50e-06&6.3e-06\\
			$\textbf{NPSOR}$&$\textbf{IT}$&12  &  13 & 14 & 14 & 14 & 14\\
			$\alpha_1=1.7$	&$\textbf{CPU} $ & 0.0017 & 0.0030 &0.0112 &0.0356&0.0921&0.1865\\
			&$\textbf{Res} $&5.5e-06 &6.6e-06 &3.6e-06&4.4e-06& 5.9e-06&7.5e-06\\
			\hline 
		\end{tabular}
	\end{table}
\end{ex}
\noindent	From Table \ref{t1} and Table \ref{t2}, we can observe that the number of iteration steps required for our proposed  NPGS and NPSOR methods is less than the MGS and MSOR methods.
\section{Conclusion} 
In this article, we introduce a class of new projected-type iteration methods based on matrix splitting for solving the linear complementarity problem LCP $(\sigma, \mathcal{A})$. During the iteration process, the large and sparse structure of $\mathcal{A}$ is maintained by these iterative forms. Moreover, the sufficient  conditions for convergence for  $H_+$ matrix or  $P$-matrix are presented. Finally, two numerical examples are provided to demonstrate the effectiveness of the proposed methods.\\ 
\noindent \textbf{Conflict of interest} The authors declare that there is no conflicts of interest.\\ 
\noindent \textbf{Acknowledgment.}  
The first author is thankful to the University Grants Commission (UGC), Government of India, under the JRF fellowship programme no. 1068/(CSIR-UGC NET DEC. 2017).
	\bibliographystyle{plain}
	\bibliography{bharat6}
\end{document}